\documentclass[a4paper,11pt,reqno]{amsart}
\usepackage{graphicx}
\usepackage{a4wide}
\usepackage{amssymb}
\usepackage{amsthm}
\usepackage{eucal}
\usepackage[pdftex,colorlinks]{hyperref}

\numberwithin{equation}{section}

\newcommand{\be}{\begin{equation}}
\newcommand{\ee}{\end{equation}}
\newcommand{\ba}{\begin{eqnarray}}
\newcommand{\ea}{\end{eqnarray}}

\newtheorem{theorem}{Theorem}[section]
\newtheorem{proposition}[theorem]{Proposition}
\newtheorem{remark}[theorem]{Remark}
\newtheorem{lemma}[theorem]{Lemma}

\begin{document}

\title [Controllability for the Keller-Segel system]{A uniform controllability result for the Keller-Segel system} 

\author[F. W. Chaves-Silva]{Felipe Wallison Chaves-Silva$^*$}
\address{BCAM-Basque Center for Applied Mathematics
Mazarredo 14, 48009 Bilbao, Basque Country, Spain}
\email{felipewallison@gmail.com}
\author[S. Guerrero]{Sergio Guerrero}
\address{Universit\'e Pierre et Marie Curie-Paris 6, UMR 7598 Laboratoire Jacques-Louis Lions, Paris,
F-75005 France}
\email{guerrero@ann.jussieu.fr}

\thanks{$^*$ Corresponding author}



\maketitle
\begin{abstract} 
In this paper we study the controllability of the Keller-Segel system approximating its parabolic-elliptic version. We show that this parabolic system is locally uniform controllable around a constant solution of the parabolic-elliptic system when the control is acting on the component of the chemical.

\vspace{0.5cm}

\noindent\textsc{R\'esum\'e.} Dans cet article, nous \'etudions la contr\^olabilit\'e du syst\`eme de Keller-Segel qui approxime  sa version parabolique-elliptique. Nous montrons que ce syst\`eme parabolique est localement uniform\'ement contr\^olable  autour d'une solution constante du syst\`eme parabolique-elliptique lorsque le contr\^ole agit sur la substance chimique.

\end{abstract}


\section{Introduction}
Let  $\Omega \subset \mathbb{R}^N$ (N = 2, 3) be a bounded connected open set whose boundary $\partial \Omega$ is regular enough. Let $T > 0$ and  $\omega'$ and $\omega$  be two (small) nonempty subsets of $\Omega$ with $\omega' \subset \subset \omega$. We will use the notation $Q =  \Omega \times (0,T) $ and  $\Sigma = \partial \Omega \times (0,T)$ and we will denote by $\nu(x)$ the outward normal to $\Omega$ at the point $x \in \partial \Omega$.

We will be concerned with the following controlled Keller-Segel system
\begin{equation}\label{system}
\left |   
\begin{array}{ll}
u_{t}  - \Delta u  = -\nabla \cdot (u\nabla v)   &     \mbox{in}  \  \ Q,  \\
\epsilon v_t - \Delta v = au - bv  + g \chi  &     \mbox{in}  \  \ Q, \\
\frac{\partial u}{\partial \nu} = \frac{\partial v}{\partial \nu} = 0     &    \mbox{on}  \  \   \Sigma, \\
u(x,0) = u_0;  \ v(x,0) = v_0   &    \mbox{in}    \  \  \Omega,
\end{array}
\right. 
\end{equation}
where $a$ and $b$  are positive real constants, $u_0, v_0 \geq 0$ are the  initial data,  $g$ is  an internal control and $\epsilon$ is a small positive parameter, which is intended to tend to zero.  In \eqref{system}, $\chi : \mathbb{R}^N \rightarrow \mathbb{R}$ is a $C^\infty$  function  such that $supp \  \chi \subset \subset \omega$, $0 \leq \chi \leq 1$ and $\chi \equiv 1$ in $\omega'$.

System \eqref{system} is a classical equation in chemotaxis, describing the change of motion when a population reacts in response  to an external chemical stimulus spread in the environment where they reside.  In many applications (see, for instance, \cite{Biler-5, LR, AR}), system \eqref{system}  is  approximated by the following parabolic-elliptic system:

\begin{equation}\label{system0}
\left |   
\begin{array}{ll}
u_{t}  - \Delta u  = -\nabla \cdot (u\nabla v)  &     \mbox{in}  \  \ Q,  \\
- \Delta v = au -bv  + g \chi &     \mbox{in}  \  \ Q, \\
\frac{\partial u}{\partial \nu} = \frac{\partial v}{\partial \nu} = 0     &    \mbox{on}  \  \   \Sigma, \\
u(x,0) = u_0    &    \mbox{in}    \  \  \Omega.
\end{array}
\right. 
\end{equation}

In \eqref{system} and \eqref{system0},  $u = u(x,t) \geq 0$ and $v = v(x,t) \geq 0$ represent, respectively,  the concentrations of species (i.e, the population density) and that of the chemical (i.e., concentration of the chemical substance). For more details about the Keller-Segel system see, for instance, \cite{Biler, Carrillo-1, Horstman, Hillen-Painter, Keller-Segel, Rascle-Ziti, Wang}.

The goal of this paper is to analyze the controllability of   \eqref{system} around a fixed trajectory of \eqref{system0}, uniformly with respect to $\epsilon$. More precisely, we consider a  constant solution  $(M_1 ,M_2) \in \mathbb{R}^2 $  of \eqref{system0}, with $g\equiv 0$ , and we seek for a control $g =g(\epsilon)$  such that $(u(T),v(T)) = (M_1,M_2)$ and $g$ is bounded with respect to $\epsilon$.

\begin{remark}
 Each one of the models \eqref{system} and \eqref{system0} can be viewed as a single nonlinear parabolic equation for $u$ with a \textit{nonlocal} (either in $x$ or $(x,t$)) nonlinearity, since the term $\nabla v$ can be expressed as a linear integral operator acting on $u$. In the first model, the variations of the concentration $v$ are governed by the linear nonhomogeneous heat equation, and therefore are slower than in the latter system, where the response of $v$ to the variations of $u$ are instantaneous, and described by the integral operator $(-\Delta)^{-1}$ whose kernel has a singularity. Thus, one may expect the evolution described by (\ref{system0}) to be faster than in (\ref{system}), especially for large values of $\epsilon$ when the diffusion of $v$ is rather slow compared to that of $u$. Moreover, the nonlinear effects for (\ref{system0}) should manifest themselves faster than for (\ref{system}) (see \cite{Biler-5}).
 \end{remark}

As usual in control theory, we study the controllability of  \eqref{system} around $(M_1,M_2)$ by first analyzing the controllability of its linearization around this trajectory, namely:
\begin{equation}\label{system-linear}
\left |   
\begin{array}{ll}
u_{t}  - \Delta u  = -M_1\Delta v  + h_1 &     \mbox{in}  \  \ Q,  \\
\epsilon v_t - \Delta v = au - bv  + g \chi +h_2  &     \mbox{in}  \  \ Q, \\
\frac{\partial u}{\partial \nu} = \frac{\partial v}{\partial \nu} = 0     &    \mbox{on}  \  \   \Sigma, \\
u(x,0) = u_0 ;  \ v(x,0) = v_0  &    \mbox{in}    \  \  \Omega,
\end{array}
\right. 
\end{equation}
where $h_1 $ and $h_2$ are  given exterior forces belonging  to an appropriate Banach space $X$ (see \eqref{X}) and  having exponential decay at $t = T$.  Our objective then will be to prove that we can find $g$ so that the solution $(u,v)$ of (\ref{system-linear}) satisfies $(u(T),v(T)) = (0,0)$ and moreover we want that the quantity $\nabla \cdot (u\nabla v)$ belongs to  $X$. Then, we employ an inverse mapping argument introduced in \cite{Ima-1} in order  to obtain the controllability of \eqref{system} around $(M_1,M_2)$.

The most important tool to prove the null controllability of the linear system \eqref{system-linear} is a \textit{ global Carleman inequality} for the solutions of  its  adjoint system, that is to say,
\begin{equation}\label{system-linear-adjoint}
\left |   
\begin{array}{ll}
-\varphi_{t}  - \Delta \varphi  = a \xi  +f_1 &     \mbox{in}  \  \ Q,  \\
-\epsilon \xi_t - \Delta \xi =  -b\xi -M_1\Delta \varphi +f_2  &     \mbox{in}  \  \ Q, \\
\frac{\partial \varphi}{\partial \nu} = \frac{\partial \xi}{\partial \nu} = 0     &    \mbox{on}  \  \   \Sigma, \\
\varphi(x,T) = \varphi_T ;  \ \xi(x,T) = \xi_T  &    \mbox{in}    \  \  \Omega, \\
\int_\Omega \varphi_T dx=0,
\end{array}
\right. 
\end{equation}
where $f_1$ and $f_2$ are arbitrary $L^2(Q)$ functions.

Actually, due to the fact that the control is acting on the second equation of \eqref{system-linear}, we need to bound global integrals of $\varphi$ and $\xi$ in terms of a local integral of $\xi$ and global integrals of $f_1$ and $f_2$.  The main difficulty when proving a Carleman inequality of this type for the solution  $(\varphi, \xi)$ of  \eqref{system-linear-adjoint} comes from the fact that the coupling in the second equation is in $\Delta \varphi$ and not in $\varphi$. In fact, the inequality we prove  will contain global terms with the $L^2$-weighted norms of $\Delta \varphi$ and $\xi$ in the left hand side, no global terms in $\varphi$, while a local integral of $\xi$ and global integrals of $f_1$ and $f_2$ will appear in its right-hand side.  

With the help of the Carleman inequality and an appropriate inverse function theorem, we will prove the following result, which is the main result of this paper.
\begin{theorem}\label{mainresultt}
Let $0 < \epsilon \leq 1$ and $(M_1,M_2) \in \mathbb{R}^2_+$ be such that $aM_1 -bM_2 =0$. Then, there exists $\gamma >0$ such that, for any $(u_0,v_0) \in H^1(\Omega) \times H^2(\Omega)$ with  $u_0, v_0 \geq 0$, satisfying $\frac{1}{|\Omega|}\int_\Omega u_0dx =M_1$, $\frac{\partial v_0}{\partial \nu} = 0$ on $\partial \Omega $ and $||(u_0-M_1, v_0-M_2)||_{H^1(\Omega) \times H^2(\Omega)} \leq \gamma$, we can find $g \in L^2(0,T;H^1(\Omega))$, with $\left \| g \right \|_{L^2(0,T;H^1(\Omega))}$ bounded independently of $\epsilon$,  such that the associated solution $(u,v)$ to \eqref{system} satisfies: 
$$
(u(T),v(T)) = (M_1,M_2) \ \text{in} \ \Omega.
$$
\end{theorem}

\begin{remark}
Note that all constant trajectories $(M_1,M_2) \in \mathbb{R}^2_+$ of \eqref{system} satisfy  $aM_1 -bM_2 =0$.  On the other hand,  condition  $\frac{1}{|\Omega|}\int_\Omega u_0dx =M_1$ in Theorem \ref{mainresultt} is necessary since the  mass of $u$ is preserved, i.e., 
\begin{equation}\label{massa}
\frac{1}{|\Omega|} \int_{\Omega}u(x,t)dx = \frac{1}{|\Omega|} \int_{\Omega}u_0(x)dx,    \ \  \forall  t>0.
 \end{equation}
\end{remark}

Let us now mention some works that have been devoted to the study of the controllability of degenerating coupled parabolic systems.  

To our knowledge, the first time that the study of the controllability of coupled parabolic systems degenerating into  parabolic-elliptic ones was analyzed was in \cite{F-M} and \cite{F-MB}, where the authors analyze the local null controllability of a nonlinear coupled parabolic system approximating a parabolic-elliptic system modeling electrical activities in a cardiac tissue. Combining Carleman inequalities and weighted energy estimates, the authors prove the stability of the control properties with respect to the degenerating parameter.

Another related work  is \cite{CGP}, where the authors analyze the null controllability of degenerating coupled parabolic systems with zero-order couplings. In there, by extending the adjoint system to a system of four equations, the authors are able show that, in general, the control properties are preserved in the limit when the degenerating parameter goes to zero.

Concerning the controllabity of the Keller-Segel system, the only result we know is the one obtained in \cite{BGG-G}, where the authors analyze the controllability of the Keller system \eqref{system}, with $\epsilon =1$, around a fixed trajectory of \eqref{system} (i.e., a solution of  \eqref{system} with $g\equiv 0$), when a control is acting on the first equation, which is not natural from the physical point of view. The authors are able to show that the Keller-Segel system is controllable around this trajectory if the trajectory has good regularity properties.  However, in their case, since the control is acting on the first equation, the problem is easier from a mathematical point view because the adjoint system of the linearization of the Keller-Segel system around the trajectory  has a zero-order coupling (see \cite{Gue-1}). Another interesting work in this subject is \cite{F-L-P}, in which the authors show that, in dimension 2, any global in time bounded solution of  system \eqref{system} converges to a single equilibrium (a stationary solution of \eqref{system})  as  the time  tends to infinity.

The paper is organized as follows: In section \ref{sec2}, we prove a Carleman inequality for the system \eqref{system-linear-adjoint}. In section \ref{sec3}, we deal with the null controllability of the linearized system \eqref{system-linear}. Finally, in section \ref{sec4}, we prove the local uniform controllability of \eqref{system} around the constant trajectory $(M_1,M_2)$.

\section{Carleman inequality}\label{sec2}
In this section we prove a suitable Carleman inequality for  the adjoint system \eqref{system-linear-adjoint}. This will provide a null controllability result for the linear system \eqref{system-linear} with an appropriate $h_1$ (see section \ref{sec3}).

Before stating the desired Carleman inequality, let us introduce several weight functions which will be useful in the sequel.  The basic weight  will be a function $\eta_0 \in C^2(\overline{\Omega})$ verifying 
$$ \eta_0(x) > 0 \ \mbox{in} \  \Omega , \  \  \eta_0 \equiv 0 \  \mbox{on} \ \partial \Omega, \   \  |\nabla \eta_0(x)| > 0 \ \forall x \in \overline{\Omega \backslash \omega_0},$$
where $\omega_0 \subset \subset \omega'$ is a nonempty open set. The existence of such a function $ \eta_0$ is proved in \cite{F-Im}.
Then, for some positive real number  $\lambda$, we introduce:
\begin{align} \label{weightfunctions1}
&\phi(x,t) = \frac{e^{\lambda \eta_0(x)}}{t^4(T-t)^4}, \ \alpha(x,t) = \frac{e^{\lambda \eta^0(x)} - e^{2\lambda||\eta_0||_{\infty}}}{t^4(T-t)^4}, \nonumber \\
&\hat{\phi}(t) = \min_{x \in \overline{\Omega}} \phi(x,t), \ \phi^*(t) = \max_{x \in \overline{\Omega}} \phi(x,t), \ \alpha^*(t) = \max_{x \in \overline{\Omega}}\alpha (x,t),\ \hat{\alpha} = \min_{x \in \overline{\Omega}}\alpha (x,t).
\end{align}
Recall that weights like $\alpha$, $\phi$, etc.  were already used in \cite{Ima-1}  in order to obtain  Carleman inequalities for the (adjoint) Stokes system (see also \cite{FC-G-P}).

Let us also introduce  the following  notation:
\begin{align}\label{ineq-epsxi}
I_{\beta}(s,\sigma; q) = & \ s^{\beta+3} \iint\limits_Q e^{2s\alpha}\phi^{\beta+3}|q|^2 dxdt + s^{\beta+1} \iint\limits_Q e^{2s\alpha}\phi^{\beta+1}|\nabla q|^2 dxdt  \nonumber \\
& + s^{\beta-1} \iint\limits_Q e^{2s\alpha}\phi^{\beta-1}(\sigma^2|q_t|^2 +\sum_{i,j=1}^N |\frac{\partial^2q }{\partial x_i \partial x_j} |^2) dxdt,
\end{align}
where $s, \beta$ and $ \sigma$ are real numbers and $q = q(x,t)$.

The following Carleman inequality holds:
\begin{lemma}\label{lemma-4}
There exist $C = C(\Omega, \omega')$ and  $\lambda_0 = \lambda_0(\Omega, \omega')$ such that,  for every $\lambda \geq \lambda_0$, there exists $s_0 = s_0(\Omega, \omega', \lambda)$ such that, for any $s \geq s_0(T^4 + T^8)$, any $q_0 \in L^2(\Omega)$ and any  $f\in L^2(\Omega)$, the weak solution to 
\begin{equation}\label{heat-neumann}
\left |   
\begin{array}{ll}
\sigma q_{t}  - \Delta q  = f  &     \mbox{in}  \  \ Q,  \\
\frac{\partial q}{\partial \nu} = 0     &    \mbox{on}  \  \   \Sigma, \\
q(x,0) = q_0&    \mbox{in}    \  \  \Omega,
\end{array}
\right. 
\end{equation}
satisfies 
\begin{align}
I_{\beta}(s,\sigma; q) & \leq C\biggl(s^\beta \iint\limits_Q e^{2s\alpha}\phi^{\beta}|f|^2 dxdt  \nonumber + s^{\beta+3} \iint\limits_{\omega' \times (0,T)} e^{2s\alpha}{\phi}^{\beta+3}|q|^2 dxdt \biggl),
\end{align}
for all $\beta \in \mathbb{R}$ and any $0 < \sigma \leq 1$.
\end{lemma}

The proof of Lemma  \ref{lemma-4} can be deduced from the Carleman inequality for the heat equation with homogeneous Neumann boundary conditions given in \cite{F-Im}.

The main result of this section is  as follows:
\begin{theorem}\label{Theo-1}
Given $0 < \epsilon \leq 1$, there exist $C = C(\Omega, \omega')$ and  $\lambda_0 = \lambda_0(\Omega, \omega')$ such that,  for every $\lambda \geq \lambda_0$, there exists $s_0 = s_0(\Omega, \omega', \lambda)$ such that, for any $s \geq s_0(T^4 + T^8)$, any $(\varphi_T, \xi_T) \in L^2(\Omega)^2$ and any $f_1, f_2 \in L^2(Q)$, the solution $(\varphi, \xi)$ of system $(\ref{system-linear-adjoint})$ satisfies 
\begin{align}\label{Carleman-final}
s^3 \iint\limits_{Q} e^{2s\alpha}\phi^3 |\Delta \varphi|^2 dxdt + I_1(\epsilon, s; \xi) \leq& \  C \biggl(  s^{18} \iint\limits_{\omega' \times (0,T)} e^{2s\alpha}\phi^{18} |\xi|^2 dxdt  \\
& + s^{10} \iint\limits_{Q} e^{2s\alpha}\phi^{10} |f_1|^2 dxdt +  s^3 \iint\limits_{Q} e^{2s\alpha}\phi^3 |f_2|^2 dxdt\biggl).\nonumber
\end{align}

\end{theorem}

\begin{proof} 

For the purpose of the proof, let $\omega_i \subset \Omega$, $i = 1, 2,3$ be such that 
$$
\omega_0 \subset \subset \omega_1\subset \subset \omega_2 \subset \subset \omega_3 \subset \subset \omega 
$$
and let $\rho_i \geq 0$, $\left (i = 1,2 \right) $ satisfies 
$$
\rho_i \in C^2_c(\omega_{i+1}), \ \rho_i =1 \ \text{in} \ \omega_i.
$$

Let us assume  for the moment that  $f_1$, $f_2 \in C_0^{\infty}(Q)$ and $\varphi_T, \xi_T \in C^{\infty}_0(\Omega)$. From $\eqref{system-linear-adjoint}_1$ (the first equation in \eqref{system-linear-adjoint}) we have that $\Delta \varphi$ satisfies 
\begin{equation}\label{Deltavarphi}
\left |   
\begin{array}{ll}
-(\Delta \varphi)_{t}  - \Delta (\Delta \varphi)  = a \Delta \xi  + \Delta f_1 &     \mbox{in}  \  \ Q,  \\
\frac{\partial  \Delta \varphi}{\partial \nu} = 0     &    \mbox{on}  \  \   \Sigma, \\
\Delta\varphi(x,T) = \Delta \varphi_T   &    \mbox{in}    \  \  \Omega. \\
\end{array}
\right. 
\end{equation}
Applying inequality \eqref{heat-neumann} to $\eqref{system-linear-adjoint}_2$ (the second equation in \eqref{system-linear-adjoint}), inequality \eqref{car-transpo} (in appendix \ref{ApA})  to \eqref{Deltavarphi}  and adding these two inequalities, we get
\begin{align}\label{1cal}
s^3 \iint\limits_{Q} e^{2s\alpha}\phi^3 |\Delta \varphi|^2 dxdt & +  I_1(\epsilon, s; \xi) \nonumber \\
 \leq &  \  C \biggl(  s^3 \iint\limits_{\omega_1 \times (0,T)} e^{2s\alpha}\phi^3 |\Delta \varphi|^2 dxdt + s^4 \iint\limits_{\omega' \times (0,T)} e^{2s\alpha}\phi^4 |\xi|^2 dxdt \nonumber \\
 & +  s^4 \iint\limits_{Q} e^{2s\alpha}\phi^4 |f_1|^2 dxdt +  s \iint\limits_{Q} e^{2s\alpha}\phi |f_2|^2 dxdt \biggl),
\end{align}
for $s \geq s_0(T^4+T^8)$.

The rest of the proof is devoted to estimate  the local integral in $\Delta \varphi$ in the right-hand side of \eqref{1cal}.  First, we observe that 
\begin{align}\label{local-1}
  s^3 \iint\limits_{\omega_1 \times (0,T)} e^{2s\alpha}\phi^3 |\Delta \varphi|^2 dxdt  & \leq   s^3 \iint\limits_{\omega_3 \times (0,T)} \rho_1 \rho_2e^{2s\alpha}\phi^3 |\Delta \varphi|^2 dxdt  \\
  & =  -\frac{s^3}{M_1} \iint\limits_{\omega_3 \times (0,T)} \rho_1 \rho_2 \phi^3 e^{2s\alpha}\Delta \varphi (-\epsilon \xi_t - \Delta \xi   + b\xi- f_2  ) dxdt \nonumber 
\end{align}
and estimate each one of the terms in the right-hand side of \eqref{local-1}.

The following estimate is straightforward
\begin{align}\label{local-15}
s^3 \iint\limits_{\omega_3 \times (0,T)} \rho_1 \rho_2 e^{2s\alpha} \phi^3 \Delta \varphi(b\xi- f_2)dxdt \leq &  \  C_{\delta}s^3  \iint\limits_{\omega_3 \times (0,T)} e^{2s\alpha} \phi^3  (|\xi|^2 +|f_2|^2) dxdt \nonumber \\
&+ \delta s^3 \iint\limits_{\omega_3 \times (0,T)}  e^{2s\alpha} \phi^3 |\Delta \varphi |^2 dxdt, 
\end{align}
for any $\delta >0$.

In order to estimate the other two terms in \eqref{local-1}, we introduce a  function $\theta = \theta(x,t)$ given by
\begin{equation}\label{theta}
\theta = \rho_1 s^3\phi^3 e^{s\alpha}.
\end{equation}
From \eqref{Deltavarphi} we see that 
\begin{equation}\label{Deltavarphi-1}
\left |   
\begin{array}{ll}
-(\theta \Delta \varphi)_{t}  - \Delta (\theta \Delta \varphi)  = a \theta \Delta \xi  + \theta \Delta f_1 - \theta_t \Delta \varphi - 2\nabla \theta \cdot \nabla (\Delta \varphi)  - \Delta \theta \Delta \varphi&     \mbox{in}  \  \ Q,  \\
\theta \Delta \varphi = 0     &    \mbox{on}  \  \   \Sigma, \\
\theta \Delta\varphi(x,T) = 0   &    \mbox{in}    \  \  \Omega. \\
\end{array}
\right. 
\end{equation}
We write $\theta \Delta \varphi = \eta + \psi$, where $\eta$ and $\psi$ solve, respectively, 
\begin{equation}\label{eta}
\left |   
\begin{array}{ll}
-\eta_{t}  - \Delta \eta = a \theta \Delta \xi  + \theta \Delta f_1&     \mbox{in}  \  \ Q,  \\
\eta = 0     &    \mbox{on}  \  \   \Sigma, \\
\eta (x,T) = 0  &    \mbox{in}    \  \  \Omega \\
\end{array}
\right. 
\end{equation}
and 
\begin{equation}\label{psi}
\left |   
\begin{array}{ll}
-\psi_{t}  - \Delta \psi =  - \theta_t \Delta \varphi - 2\nabla \theta \cdot \nabla (\Delta \varphi)  - \Delta \theta \Delta \varphi&     \mbox{in}  \  \ Q,  \\
\psi= 0     &    \mbox{on}  \  \   \Sigma, \\
\psi(x,T) = 0   &    \mbox{in}    \  \  \Omega. \\
\end{array}
\right. 
\end{equation}
We have 
\be\label{Deltaxi}
s^3 \iint\limits_{\omega_3 \times (0,T)} \rho_1\rho_2 e^{2s\alpha} \phi^3 \Delta \varphi \Delta \xi dxdt =  \iint\limits_{\omega_3 \times (0,T)} \rho_2e^{s\alpha} (\eta +\psi) \Delta \xi dxdt.
\ee
The first term in the right-hand side of \eqref{Deltaxi} can be estimated as follows
\be\label{local-11}
 \iint\limits_{\omega_3 \times (0,T)} \rho_2e^{s\alpha} \eta \Delta \xi dxdt \leq \delta \iint\limits_{\omega_3 \times (0,T)} e^{2s\alpha}|\Delta \xi|^2 dxdt + C_{\delta}s^{10} \iint\limits_{\omega_2 \times (0,T)} \phi^{10} e^{2s\alpha}(|\xi|^2+|f_1|^2) dxdt.
\ee
In fact, to prove \eqref{local-11},  we use following estimate: 
\begin{lemma} \label{estimate-1}
The solution $\eta$ of \eqref{eta} satisfies
\be\label{local-x1}
 \iint\limits_{Q}  |\eta|^2 dxdt \leq  C s^{10} \iint\limits_{\omega_2 \times (0,T)} e^{2s\alpha} \phi^{10}(|\xi|^2+|f_1|^2) dxdt,
\ee
for a constant $C>0$.
\end{lemma}
We prove estimate \eqref{local-x1} at the end of appendix \ref{ApA}.

%
%

For the second term in the right-hand side of \eqref{Deltaxi}, we use integration by parts to get
\begin{align}\label{local-2}
 \iint\limits_{\omega_3 \times (0,T)} \rho_2e^{s\alpha} \psi \Delta \xi dxdt =& - \iint\limits_{\omega_3 \times (0,T)} \rho_2 s^{7/2}\phi^{7/2} e^{s\alpha} \nabla \xi \cdot \nabla \bigl( \frac{\psi}{s^{7/2}\phi^{7/2} } \bigl)dxdt  \nonumber \\
& - \iint\limits_{\omega_3 \times (0,T)} \nabla (\rho_2 s^{7/2}  \phi^{7/2} e^{s\alpha}) \cdot \nabla \xi \cdot  \frac{\psi}{s^{7/2}\phi^{7/2} } dxdt \nonumber \\
& \leq C_{\delta}\iint\limits_{\omega_3 \times (0,T)} s^{9}\phi^{9} e^{2s\alpha} |\nabla \xi|^2dxdt   +  \delta \left \| \frac{\psi}{s^{7/2}\phi^{7/2} }\right \|^2_{L^2(0,T; H^1_0(\Omega))}, 
\end{align}
for any $\delta >0$.

For the sequel, we need the following result:
\begin{lemma}\label{lemmaenerg-1} The solution $\psi$ of \eqref{psi} can be estimated as follows:
\begin{align}\label{energorpsi-1}
\left \| \frac{\psi}{s^{7/2}\phi^{7/2}}\right \|^2_{L^2(0,T; H^1_0(\Omega))} & +  \left \| \left (\frac{\psi}{s^{7/2}\phi^{7/2}}\right )_t\right \|^2_{L^2(0,T; H^{-1}(\Omega))}  \nonumber \\
&\leq   C \biggl ( \iint\limits_{Q} |\eta|^2 dxdt + s^{3} \iint\limits_{Q} \phi^{3} e^{2s\alpha}|\Delta \varphi|^2 dxdt \biggl), 
\end{align}
for a constant $C>0$ independent of $s$.
\end{lemma}
In order to prove Lemma \ref{lemmaenerg-1} , we consider the system satisfied by $\frac{\psi}{s^{7/2}\phi^{7/2}}$ and we perform standard energy estimates. This yields the $L^2(0,T; H^1_0(\Omega))$ estimate. Then, the $L^2(0,T; H^{-1}(\Omega))$ estimate  is a direct consequence from the fact that 
$$
\left \| \Delta \left (\frac{\psi}{s^{7/2}\phi^{7/2}}\right )\right \|_{L^2(0,T; H^{-1}(\Omega))} \leq C \left \| \frac{\psi}{s^{7/2}\phi^{7/2}}\right \|_{L^2(0,T; H^{1}_0(\Omega))}.
$$
For the sake of simplicity, we omit the complete proof.

From \eqref{local-2} and  Lemma \ref{lemmaenerg-1}, it follows that 
\begin{align}\label{local-1-1}
 \iint\limits_{\omega_3 \times (0,T)} \rho_2 e^{s\alpha} \psi \Delta \xi dxdt   \leq & \ C \iint\limits_{\omega_3 \times (0,T)} s^{9}\phi^{9} e^{2s\alpha} |\nabla \xi|^2dxdt \nonumber \\ 
& + \delta \biggl(  \iint\limits_{Q} |\eta|^2 dxdt + s^3 \iint\limits_{Q}  \phi^3 e^{2s\alpha}|\Delta \varphi|^2 dxdt\biggl). 
\end{align} 
Hence, from \eqref{energorpsi-1} and  \eqref{local-1-1}, the term in \eqref{Deltaxi} is estimated as follows:
\begin{align}\label{Deltaxi-1}
s^3 \iint\limits_{\omega_3 \times (0,T)} \rho_1\rho_2 e^{2s\alpha} \phi^3 \Delta \varphi \Delta \xi dxdt  \leq & \ C \biggl( \iint\limits_{\omega_3 \times (0,T)} s^{9}\phi^{9} e^{2s\alpha} |\nabla \xi|^2dxdt \nonumber \\
& + s^{10} \iint\limits_{\omega_2 \times (0,T)} \phi^{10} e^{2s\alpha}(|\xi|^2+|f_1|^2) dxdt \biggl) \\
& + \delta  \biggl(  \iint\limits_{\omega_3 \times (0,T)} e^{2s\alpha}|\Delta \xi|^2 dxdt + s^3 \iint\limits_{Q}  \phi^3 e^{2s\alpha}|\Delta \varphi|^2 dxdt \biggl).  \nonumber 
\end{align}

Let us now estimate the first term in the right-hand side of \eqref{local-1}. We have 
\begin{align}\label{xit}
 \epsilon s^{3} \iint\limits_{\omega_3 \times (0,T)} \rho_1\rho_2 e^{2s\alpha} \phi^{3}\Delta \varphi \xi_t dxdt =& \  \epsilon \iint\limits_{\omega_3 \times (0,T)} \rho_2 e^{s\alpha}(\eta + \psi)\xi_t dxdt.  
 \end{align}
It is immediate that 
\begin{align}
 \epsilon \iint\limits_{\omega_3 \times (0,T)} \rho_2 e^{s\alpha}\eta\xi_t dxdt \leq  C_\delta  \iint\limits_{\omega_3 \times (0,T)} |\eta|^2dxdt +  \delta \epsilon^2 \iint\limits_{\omega_3 \times (0,T)} e^{2s\alpha}|\xi_t|^2dxdt,   
\end{align}
for any $\delta >0$.

For the other term in \eqref{xit}, we have 
\begin{align}
\epsilon \iint\limits_{\omega_3 \times (0,T)} \rho_2 e^{s\alpha}\psi\xi_t dxdt =& \  \epsilon \biggl< \left (\frac{\psi}{s^{7/2}\phi^{7/2}}\right)_t,e^{s\alpha}s^{7/2}\phi^{7/2}\xi \rho_2 \biggl>_{L^2(0,T;H^{-1}(\Omega), L^2(0,T;H^{1}_0(\Omega))} \nonumber \\
&+ \iint\limits_{\omega_3 \times (0,T)}\rho_2\frac{\psi}{s^{7/2}\phi^{7/2}}(e^{s\alpha}s^{7/2}\phi^{7/2})_t \xi dxdt     \nonumber \\
 \leq & \  \delta \epsilon^2 \left \|\left (\frac{\psi}{s^{7/2}\phi^{7/2}}\right)_t\right \|^2_{L^2(0,T; H^{-1}(\Omega))} + C_{\delta } \left \|s^{7/2}\phi^{7/2} e^{s\alpha} \rho_ 2\xi\right \|^2_{L^2(0,T; H^{1}_0(\Omega))}  \nonumber \\
 & +  \delta \left \| \frac{\psi}{s^{7/2}\phi^{7/2}}\right \|^2_{L^2(Q)} + C_{\delta} \iint\limits_{\omega_3 \times (0,T)}|(e^{s\alpha}s^{7/2}\phi^{7/2})_t|^2 |\xi|^2 dxdt.
\end{align}

Therefore, from Lemma \ref{lemmaenerg-1}, we obtain 
\begin{align}\label{local-10}
 \epsilon s^{3} \iint\limits_{\omega_3 \times (0,T)} \rho_1\rho_2 e^{2s\alpha} \phi^{3}\Delta \varphi \xi_t dxdt  \leq & \ C \biggl( \iint\limits_{\omega_3 \times (0,T)} |\eta|^2dxdt + \left \|s^{7/2}\phi^{7/2} e^{s\alpha} \rho_ 2\xi\right \|^2_{L^2(0,T; H^{1}_0(\Omega))}   \nonumber \\
&+ \iint\limits_{\omega_3 \times (0,T)}|(e^{s\alpha}s^{7/2}\phi^{7/2})_t|^2 |\xi|^2 dxdt \biggl)  +  \delta \biggl( \epsilon^2 \iint\limits_{\omega_3 \times (0,T)} e^{2s\alpha}|\xi_t|^2dxdt  \nonumber \\
 & +s^{3} \iint\limits_{Q} \phi^{3} e^{2s\alpha}|\Delta \varphi|^2 dxdt\biggl) .
\end{align}

Hence, from \eqref{local-15}, \eqref{Deltaxi-1} and  \eqref{local-10}, we get
\begin{align}\label{local-14}
  s^3 \iint\limits_{\omega_1 \times (0,T)} e^{2s\alpha}\phi^3 |\Delta \varphi|^2 dxdt  \leq &\  C\biggl(s^3  \iint\limits_{\omega_2 \times (0,T)} e^{2s\alpha} \phi^3  (|\xi|^2 +|f_2|^2) dxdt  +   \iint\limits_{\omega_3 \times (0,T)} s^{9}\phi^{9} e^{2s\alpha} |\nabla \xi|^2dxdt \nonumber \\
 & +   \left \|s^{7/2}\phi^{7/2} e^{s\alpha} \rho_ 2\xi\right \|^2_{L^2(0,T; H^{1}_0(\Omega))} + \iint\limits_{\omega_3 \times (0,T)}|(e^{s\alpha}s^{7/2}\phi^{7/2})_t|^2 |\xi|^2 dxdt  \nonumber \\
&+  s^{10} \iint\limits_{\omega_2 \times (0,T)} e^{2s\alpha} \phi^{10}(|\xi|^2+|f_1|^2) dxdt \biggl) \nonumber \\
& + \delta \biggl( \iint\limits_{\omega_3 \times (0,T)} e^{2s\alpha}|\Delta \xi|^2 dxdt + s^3 \iint\limits_{Q} \phi^3 e^{2s\alpha}|\Delta \varphi|^2 dxdt \nonumber \\
&+ \epsilon^2 \iint\limits_{\omega_3 \times (0,T)} e^{2s\alpha} |\xi_t|^2dxdt\biggl), 
\end{align}
for any $\delta >0$.

Combining \eqref{local-14} and \eqref{1cal}, and taking $\delta >0$ small enough, we obtain
\begin{align}\label{1cal-1}
s^3 \iint\limits_{Q} e^{2s\alpha}\phi^3 |\Delta \varphi|^2 dxdt + I_1(\epsilon, s; \xi) \leq& \  C \biggl(  s^4 \iint\limits_{\omega' \times (0,T)} e^{2s\alpha}\phi^4 |\xi|^2 dxdt  + \iint\limits_{\omega_3 \times (0,T)} s^{9}\phi^{9} e^{2s\alpha}(|\xi|^2 + |\nabla \xi|^2)dxdt\nonumber \\
& + s^{10} \iint\limits_{Q} e^{2s\alpha}\phi^{10} |f_1|^2 dxdt +  s^3 \iint\limits_{Q} e^{2s\alpha}\phi^3 |f_2|^2 dxdt \biggl ).
\end{align}
To finish the proof, we estimate the local integrals involving $\nabla \xi$ in \eqref{1cal-1}.

Integration by parts gives
\begin{align}\label{2.26}
 \iint\limits_{\omega_3 \times (0,T)} s^{9}\phi^{9} e^{2s\alpha} |\nabla \xi|^2dxdt  \leq &   \iint\limits_{\omega' \times (0,T)} \rho s^{9}\phi^{9} e^{2s\alpha} |\nabla \xi|^2dxdt \nonumber \\
 = & \  - \iint\limits_{\omega' \times (0,T)} \rho s^{9}\phi^{9} e^{2s\alpha} \Delta \xi \xi  dxdt  \nonumber \\
 & -  \iint\limits_{\omega' \times (0,T)} \nabla(s^{9}\phi^{9} e^{2s\alpha}\rho ) \cdot \nabla \xi \xi    dxdt,
\end{align}
where $ \rho \in C^2_c(\omega')$ is such that 
$$
\rho \geq 0 \ \text{and} \  \rho =1 \ \text{in} \ \omega_3.
$$

From \eqref{1cal-1} and \eqref{2.26}, we obtain 
\begin{align}\label{1cal-2}
s^3 \iint\limits_{Q} e^{2s\alpha}\phi^3 |\Delta \varphi|^2 dxdt + I_1(\epsilon, s; \xi) \leq& \  C \biggl(  s^{18} \iint\limits_{\omega' \times (0,T)} e^{2s\alpha}\phi^{18} |\xi|^2 dxdt  \\
& + s^{10} \iint\limits_{Q} e^{2s\alpha}\phi^{10} |f_1|^2 dxdt +  s^3 \iint\limits_{Q} e^{2s\alpha}\phi^3 |f_2|^2 dxdt\biggl). \nonumber
\end{align}

Using the density of $C^\infty_0(Q)$ and $C^\infty_0(\Omega)$ in $L^2(Q)$ and $L^2(\Omega)$, respectively, we finish the proof of Theorem  \ref{Theo-1}.

%



\end{proof}

%


\section{Null controllability of the linear system with a right-hand side}\label{sec3}

In this section we want to solve the null controllability problem for the system \eqref{system-linear} with a right-hand side which decays exponentially as $t \rightarrow T^-$. 

This result will be crucial when proving the local  controllability of \eqref{system} in the next section.

Indeed, for any $0 < \epsilon \leq 1$, we would like to find a control $g = g(\epsilon)$, bounded independently of $\epsilon$,  such that the solution to 

\begin{equation}\label{system-linear-4}
\begin{cases}
L(u,v) = (h_1, g \chi + h_2)&     \mbox{in}  \  \ Q, \\
\frac{\partial u}{\partial \nu} = \frac{\partial v}{\partial \nu} = 0     &    \mbox{on}  \  \   \Sigma, \\
u(x,0) = u_0;  \ v(x,0) = v_0   &    \mbox{in}    \  \  \Omega, \\
\end{cases}
\end{equation}
where
\begin{equation}\label{5.2}
 L(u,v) = (u_{t}  - \Delta u  + M_1\Delta  v, \epsilon v_t - \Delta v +bv -au),
 \end{equation}
satisfies
\begin{equation}\label{5.3}
u(x,T) = 0;  \ v(x,T) = 0   \     \mbox{in}     \  \Omega.
\end{equation}
Furthermore, it will be convenient to prove the existence of a solution of the previous problem in an appropriate weighted space. Before introducing the spaces where we solve  problem \eqref{system-linear-4}-\eqref{5.3}, we improve the Carleman inequality obtained in the previous section. This Carleman inequality will contain only weight functions that do not vanish at $t=0$.  In order to introduce these new weights, let us consider the function
	
\begin{equation}\label{system-linear-8}
l(t) = \left\{ 
\begin{array}{llcc}
(T^2/4)  & \text{if} \ 0 \leq t \leq T/2    \\
t(T-t) & \text{if} \  T/2 \leq t \leq T.
\end{array}
\right. 
\end{equation}
and  we define our new weight functions to be
$$
\beta(x,t) = \frac{e^{\lambda \eta^0(x)} - e^{2\lambda||\eta_0||_{\infty}} }{l^4(t)}, \  \gamma(x,t) = \frac{e^{\lambda \eta_0(x)}}{l^4(t)},
$$
\begin{align} \label{weightfunctions1}
\hat{\gamma}(t) = \min_{x \in \overline{\Omega}} \gamma(x,t), \ \gamma^*(t) = \max_{x \in \overline{\Omega}} \phi(x,t), \ \beta^*(t) = \max_{x \in \overline{\Omega}}\beta (x,t),\ \hat{\beta} = \min_{x \in \overline{\Omega}}\beta (x,t).
\end{align}

With these new weights, we can state our refined Carleman estimate as follows:

\begin{lemma}
Given $0 < \epsilon \leq 1$,  there exists a positive constant $C$ depending on $T$, $s$ and $\lambda$, such that every solution of \eqref{system-linear-adjoint} verifies:
\begin{align}\label{C-8}
& \iint\limits_Q e^{2s\beta}\gamma^4|\xi|^2dxdt+ \iint\limits_Q  e^{2s\beta}\gamma^{2} |\nabla \xi|^2dxdt + \iint\limits_Q e^{2s\hat{\beta}} \hat{\gamma}^{3}| \varphi - \bigl(\varphi\bigl)_\Omega |^2dxdt \nonumber \\
& + \iint\limits_Q e^{2s\hat{\beta}}\hat{\gamma}^{3} |\nabla \varphi|^2dxdt  +  \left \| \varphi(.,0) - \bigl(\varphi(.,0)\bigl)_\Omega \right \|^2_{L^2(\Omega)} + \epsilon \left \| \xi(.,0)\right \|^2_{L^2(\Omega)}  \nonumber \\
&\leq C\biggl( \iint\limits_Q e^{2s\beta^*} (\gamma^*)^{10}|f_1|^2dxdt  +  \iint\limits_Q e^{ 2 s\beta^*}(\gamma^*)^{3} |f_2|^2dxdt +  \iint\limits_Q e^{2s\beta^*}(\gamma^*)^{18} |\chi|^2|\xi|^2dxdt\biggl),
\end{align}
where 
$$
 \bigl(\varphi\bigl)_\Omega(t) = \frac{1}{|\Omega|}\int_\Omega \varphi(x,t)dx.
$$
\end{lemma} 
The proof of this lemma is standard. It combines energy estimates, together with the fact that $\beta \leq \alpha$ in $Q$.

Now, we proceed to the definition of the spaces where \eqref{system-linear-4}-\eqref{5.3} will be solved. The main space will be:

\begin{align*}
E = \biggl \{&  (u,v,g) \in E_0:  e^{-s\hat{\beta}}\hat{\gamma}^{-3/2} \bigl ( L(u,v) \bigl )_1 \in L^2(Q),  e^{-s\beta}\gamma^{-1}  \biggl(  \bigl ( L(u,v) \bigl)_2 -g\chi \biggl ) \in  L^2(0,T; H^1(\Omega)),    \\ 
& \int_\Omega \bigl(L(u,v)\bigl)_1 dx = 0  \ \text{and} \ \frac{\partial u}{\partial \nu} = \frac{\partial v}{\partial \nu} =0  \ \text{on} \ \Sigma \biggl \},
\end{align*}
where 
\begin{align*}
E_0 = \biggl\{ & (u,v,g):  e^{-s\beta^*} (\gamma^*)^{-5}u,  e^{-s\beta^*}(\gamma^*)^{-3/2} v,  \chi e^{-s\beta^*}(\gamma^*)^{-9}g  \in  L^2(Q), \\ 
& e^{s/2\beta^* - s\hat{\beta}} \hat{\gamma}^{13/8}u \in L^{2}(0,T; H^2(\Omega)) \cap  L^{\infty}(0,T; H^1(\Omega)), \  e^{-s/2\beta^*} \hat{\gamma}^{-25/8} v \in L^2(0,T; H^3(\Omega)) \nonumber \\
& \text{and} \  e^{-s/2\beta^*} \hat{\gamma}^{-25/8} g \in L^2(0,T; H^1(\Omega)) \biggl\}.
\end{align*}

Observe that $E$ is a Banach space for the norm:
\begin{align}
||(u,v,g)||_E = & \  \left \|e^{-s\beta^*} (\gamma^*)^{-5}u\right \|_{L^2(Q)} + \left \| e^{-s\beta^*}(\gamma^*)^{-3/2} v\right \|_{L^2(Q)} + \left \| \chi e^{-s\beta^*}(\gamma^*)^{-9}g\right \|_{L^2(Q)} \nonumber \\
& +  \left \| e^{-s\hat{\beta}}\hat{\gamma}^{-3/2} \bigl( L(u,v) \bigl)_1\right \|_{L^2(Q)} + \left\| e^{-s\beta}\gamma^{-1}  \biggl ( \bigl(L(u,v)\bigl)_2 -g\chi \biggl ) \right \|_{L^2(0,T;H^1(\Omega))} \nonumber \\  
 & +  \left \| e^{s/2\beta^* - s\hat{\beta}} \hat{\gamma}^{13/8}u\right \|_{L^{2}(0,T; H^2(\Omega)) \cap L^{\infty}(0,T; H^1(\Omega))}  +  \left \| e^{-s/2\beta^*} \hat{\gamma}^{-25/8} v\right \|_{L^2(0,T; H^3(\Omega))}  \nonumber \\
& +  \left \|e^{-s/2\beta^*} \hat{\gamma}^{-25/8} g \right \|_{L^2(0,T; H^1(\Omega))}.
\end{align}

%

\begin{remark}
If $(u,v,g) \in E$, then $u(T)=v(T) =0$, so that $(u,v,g)$ solves a null controllability problem for system \eqref{system-linear} with an appropriate right-hand side $(h_1,h_2)$.
\end{remark}

We have the following result:
\begin{proposition}\label{UnifResult}
Let $0 < \epsilon \leq 1$ and  $(M_1, M_2) \in  \mathbb{R}^2$ be such that  $aM_1-bM_2=0$. Assume that:
\begin{align}
(u_0,v_0)\in H^1(\Omega)\times H^2(\Omega), \ \int_\Omega u_0dx=0, \ \frac{\partial v_0}{\partial \nu} =0 \ \text{on} \ \partial \Omega
\end{align}
and
\begin{align}
e^{-s\hat{\beta}}\hat{\gamma}^{-3/2}h_1 \in L^2(0,T; L^2_0(\Omega)),e^{-s\beta}\gamma^{-1} h_2 \in L^2(0,T; H^1(\Omega)) .
\end{align}
Then, there exists a control $g \in L^2( 0,T; H^1(\Omega))$, bounded independently of $\epsilon$, such that, if $(u,v)$ is the associated solution to \eqref{system-linear-4}, one has $(u,v, g) \in E$. In particular, \eqref{5.3} holds.
\end{proposition}

\begin{proof}
In this proof, we follow the ideas of \cite{F-Im}. 

Let $L^*$ be the adjoint operator of $L$, i.e.,
$$
L^*(z,w) = (-z_{t}  - \Delta z - aw , -\epsilon w_t - \Delta w + bw  + M_1\Delta z)
$$
and  let us introduce the space
 $$
P_0 = \biggl \{ (z,w) \in C^{\infty}(\overline{Q}); \ \frac{\partial z}{\partial \nu} = \frac{\partial w}{\partial \nu}  = 0  \ \mbox{on} \ \Sigma , \int_\Omega z(x,T)dx = 0  \  \forall t  \in [0,T] \biggl \},
 $$
 Then, for $(\zeta, \rho), (z,w) \in P_0$, we define  
 \begin{align*}
 a\bigl((\zeta, \rho), (z,w)  \bigl) := &  \ \iint\limits_Q e^{2s\beta^*} (\gamma^*)^{10} \bigl(L^*(\zeta, \rho)\bigl)_1 \bigl( L^*(z,w) \bigl)_1 dxdt \\ 
 & + \iint\limits_Q e^{2s\beta^*}(\gamma^*)^{3} \bigl(L^*(\zeta, \rho)\bigl)_2\bigl( L^*(z,w) \bigl)_2dxdt  \\ 
 & + \iint\limits_Q|\chi|^2e^{2s\beta^*}(\gamma^*)^{18} \rho w dxdt.
 \end{align*}
 From the Carleman inequality \eqref{C-8} applied to functions of $P_0$, it follows that we have a  unique continuation  property for the system 
\begin{equation}\label{5.1-1}
\begin{cases}
L^*(z,w) = (0,0)&     \mbox{in}  \  \ Q, \\
\frac{\partial z}{\partial \nu} = \frac{\partial w}{\partial \nu} = 0     &    \mbox{on}  \  \   \Sigma,
\end{cases}
\end{equation}
which implies that $a(.,.)$ is a scalar product on $P_0$. 

Therefore, we can consider the space $P$,  the completion of $P_0$ with respect to the norm associated to $a(.,.)$ (which we denote by $||.||_P$). This is a Hilbert space  and $a(.,.)$ is a continuous and coercive bilinear form on $P$.

Let us also introduce $l$, given by
 \begin{align}\label{l}
 <l, (z,w)> = \iint\limits_Qh_1zdxdt  +  \iint\limits_Qh_2wdxdt  + \int_{\Omega} u_0(x)z(x,0)dx + \epsilon \int_{\Omega} v_0(x)w(x,0)dx, 
 \end{align}
 for all $(z,w) \in P$.
 
After a simple computation, and thanks to \eqref{C-8}, we see that 
\begin{align}\label{linearform}
|<l,(z,w)>| \leq & C\biggl(  \left\| e^{-s\hat{\beta}}\hat{\gamma}^{-3/2}h_1\right \| _{L^2(Q)} +  \left\| e^{-s\beta}\gamma^{-1}h_2\right \| _{L^2(0,T;H^1(\Omega))}   \nonumber \\
& +\left\| (u_0,\epsilon^{1/2} v_0)\right \|_{L^2(\Omega)^2} \biggl) \left\| (z,w) \right \|_P   \ \ \ \ \forall (z,w) \in P.
\end{align}
In other words, $l$ is a bounded linear form on $P$ and the constant $C$ in \eqref{linearform} does not depend on $\epsilon$. Consequently, in view of Lax-Milgram's lemma, there exists a unique $(\hat{z}, \hat{w}) \in P$ satisfying:
\begin{align}\label{sol-LM}
a((\hat{z}, \hat{w}), (z,w)) = \bigl<l,(z,w)\bigl>   \ \ \ \ \forall (z,w) \in P.
\end{align}

We set
\begin{equation}\label{dual-1}
(\hat{u}, \hat{v}) =( e^{2s\beta^*} (\gamma^*)^{10}\bigl( L^*(\hat{z}, \hat{w}) \bigl)_1,e^{2s\beta^*}(\gamma^*)^{3}\bigl(L^*(\hat{z}, \hat{w})\bigl)_2  \ \text{and} \ \hat{g} = -e^{2s\beta^*}(\gamma^*)^{18} \hat{w} \chi.
\end{equation}
We must see that $(\hat{u}, \hat{v})$ satisfies:
\be\label{3.35}
 \iint\limits_Q e^{-2s\beta^*} (\gamma^*)^{-10}  |\hat{u}|^2  +  \iint\limits_Q e^{-2s\beta^*}(\gamma^*)^{-3} |\hat{v}|^2 +  \iint\limits_Q  e^{-2s\beta^*}(\gamma^*)^{-18}|\chi|^2 |\hat{g}|^2 < \infty
\ee
and that it is a solution of the reaction-diffusion system \eqref{system-linear-4}.

The first property  follows from the fact that $(\hat{z}, \hat{w}) \in P$ and 
$$
 \iint\limits_Q  e^{-2s\beta^*} (\gamma^*)^{-10}|\hat{u}|^2  + \iint\limits_Qe^{-2s\beta^*}(\gamma^*)^{-3} |\hat{v}|^2 + \iint\limits_Q e^{-2s\beta^*}(\gamma^*)^{-18}|\chi|^2 |\hat{g}|^2 = a((\hat{z}, \hat{w}),(\hat{z}, \hat{w})).
$$
In particular, from this last identity we see that  $(\hat{u}, \hat{v})  \in L^2(Q)^2$ and $ \hat{g} \in L^2(Q)$ and, from \eqref{linearform} and \eqref{sol-LM},  follows that  $\hat{g}$ is bounded independently of $\epsilon$. 

Now we consider $(\tilde{u}, \tilde{v})$ the weak solution of 
\begin{equation}\label{neweq}
\begin{cases}
\tilde{u}_{t}  - \Delta \tilde{u}  = -M_1\Delta  \tilde{v}  +  h_1 &     \mbox{in}  \  \ Q,  \\
\epsilon \tilde{v}_t - \Delta \tilde{v} + bv= au  + \hat{g} \chi +h_2  &     \mbox{in}  \  \ Q, \\
\frac{\partial \tilde{u}}{\partial \nu} = \frac{\partial \tilde{v}}{\partial \nu} = 0     &    \mbox{on}  \  \   \Sigma, \\
\tilde{u}(x,0) = u_0;  \ \tilde{v}(x,0) = v_0   &    \mbox{in}    \  \  \Omega.
\end{cases}
\end{equation}

We have that  $(\tilde{u}, \tilde{v})$ is also the unique solution of \eqref{neweq} defined by transposition.  Of course, this means that  $(\tilde{u}, \tilde{v})$  is the unique function such that 
\begin{align}\label{deftrans}
 \iint\limits_Q (\tilde{u}, \tilde{v}) \cdot (F_1,F_2)dxdt & = \iint\limits_Q h_1\varphi dxdt + \iint\limits_Q h_2\xi dxdt  + \iint\limits_Q g\chi \xi dxdt \nonumber \\
 &+\int_{\Omega} u_0(x)\varphi(x,0)dx + \epsilon \int_{\Omega} v_0(x)\xi(x,0)dx,
\end{align}
for any $(F_1,F_2) \in L^2(Q)^2$,  where $(\varphi, \xi)$ is the solution of 
\begin{equation}\label{transpo}
\left |   
\begin{array}{ll}
-\varphi_{t}  - \Delta \varphi  = a \xi  +F_1 &     \mbox{in}  \  \ Q,  \\
-\epsilon \xi_t - \Delta \xi =  -b\xi -M_1\Delta \varphi +F_2  &     \mbox{in}  \  \ Q, \\
\frac{\partial \varphi}{\partial \nu} = \frac{\partial \xi}{\partial \nu} = 0     &    \mbox{on}  \  \   \Sigma, \\
\varphi(x,T) = 0 ;  \ \xi(x,T) = 0  &    \mbox{in}    \  \  \Omega.
\end{array}
\right. 
\end{equation}

From \eqref{sol-LM} and \eqref{dual-1}, we see that $ (\hat{u}, \hat{v})$ also satisfies \eqref{deftrans}. Consequently, $ (\hat{u}, \hat{v}) = (\tilde{u}, \tilde{v})$ and $ (\hat{u}, \hat{v})$ is the solution of \eqref{system-linear-4}.

Finally, we must see that $(\hat{u}, \hat{v}, \hat{g})$ belongs to $E$.  From \eqref{3.35}, it only remains to check that
\begin{align*}
e^{s/2\beta^* - s\hat{\beta}} \hat{\gamma}^{13/8} \hat{u}  \in L^{2}(0,T; H^2(\Omega)) \cap  L^{\infty}(0,T; H^1(\Omega))
\end{align*}
and
\begin{align*}
 e^{-s/2\beta^*} \hat{\gamma}^{-25/8} (  \hat{v}, \hat{g})  \in L^2(0,T; H^3(\Omega)) \times L^2(0,T; H^1(\Omega)).
\end{align*}

To this end, let us introduce the  pair  $(u^*,v^*) = \rho(t)(\hat{u},\hat{v})$, which satisfies:

\begin{equation}\label{weightsys}
\begin{cases}
u^*_{t}  - \Delta u^*  = -M_1\Delta  v^*  +  \rho h_1 +\rho_t \hat{u} &     \mbox{in}  \  \ Q,  \\
\epsilon v^*_t - \Delta v ^*+ bv^*= au^*  + \rho \hat{g} \chi + \rho h_2  +\epsilon \rho_t \hat{v} &     \mbox{in}  \  \ Q, \\
\frac{\partial u^*}{\partial \nu} = \frac{\partial v^*}{\partial \nu} = 0     &    \mbox{on}  \  \   \Sigma, \\
u^*(x,0) = \rho(0)u_0(x);  \ v^*(x,0) = \rho(0) v_0(x)   &    \mbox{in}    \  \  \Omega. 
\end{cases}
\end{equation}

We will consider then two cases:

\textit{Case 1}. $\rho = e^{s/2\beta^* - s\hat{\beta}} \hat{\gamma}^{13/8}$.

In this case, it is not difficult to show that 
\begin{align}
|\rho_{t}| \leq  Ce^{-s\beta^*} (\gamma^*)^{-3/2} 
\end{align}
and then we have that $\rho_t \hat{u}$  and $\rho_{t} \hat{v}$ belong to $L^2(Q)$. Therefore, from well-known regularity properties of  parabolic systems (see, for instance, \cite{L-S-U}), we have
\begin{equation}
\begin{cases}
e^{s/2\beta^* - s\hat{\beta}} \hat{\gamma}^{13/8} \hat{u} \in L^{2}(0,T;H^2(\Omega)) \cap  L^\infty(0,T; H^1(\Omega)), \\
e^{s/2\beta^* - s\hat{\beta}} \hat{\gamma}^{13/8} \hat{v} \in L^2(0,T;H^2(\Omega)).
\end{cases}
\end{equation}

\textit{Case 2}. $\rho =e^{-s/2\beta^*} \hat{\gamma}^{-25/8} $.

In this case, a simple calculation gives 
\begin{align}
|\rho_{t}|  \leq Ce^{s/2\beta^* - s\hat{\beta}} \hat{\gamma}^{13/8}.
\end{align}
 Using the regularity obtained in case $1$, we conclude that   $\rho_t \hat{u}$ and $\rho_t \hat{v}$ belongs to $L^2(0,T; H^1(\Omega))$.

Using the definition of $\hat{g}$ and \eqref{C-8}, we can also show that 
\begin{align}
 \iint\limits_Q |\nabla (e^{-s/2\beta^*} \hat{\gamma}^{-25/8} \hat{g})|^2 \leq Ca((\hat{z}, \hat{w}),(\hat{z}, \hat{w})),
\end{align}
where $C$ does not depend on $\epsilon$ and hence it follows that  $e^{-s/2\beta^*} \hat{\gamma}^{-25/8} \hat{g} \in L^2(0,T; H^1(\Omega))$ and is bounded independently of $\epsilon$. 

Therefore,  from the regularity theory  for parabolic systems and Remark \ref{regulaelip} below,   we deduce that  
\begin{equation}
\begin{cases}
e^{-s/2\beta^*} \hat{\gamma}^{-25/8}  u \in L^{\infty}(0,T;H^1(\Omega)) \cap L^2(0,T;H^2(\Omega)), \\
e^{-s/2\beta^*} \hat{\gamma}^{-25/8}  \hat{v} \in L^2(0,T;H^3(\Omega)). 
\end{cases}
\end{equation}
This finishes the proof of Proposition \ref{UnifResult}.
\end{proof}

\begin{remark}\label{regulaelip}
Given any  $\epsilon >0$,  any $f \in L^2(0,T; H^1(\Omega))$ and any $z_0 \in H^2(\Omega)$, with $\frac{\partial z_0}{\partial \nu} =0$,  the solution of 
\begin{equation}\label{sus-regu}
\left |   
\begin{array}{ll}
\epsilon z_t - \Delta z +z = f  &     \mbox{in}  \  \ Q, \\
\frac{\partial z}{\partial \nu} = 0     &    \mbox{on}  \  \   \Sigma, \\
z(x,0) = z_0 &    \mbox{in}    \  \  \Omega
\end{array}
\right. 
\end{equation}
satisfies
$$
\left\| z\right \|_{L^2(0,T; H^3(\Omega))} \leq C\bigl(\left\| f \right \|_{L^2(0,T; H^1(\Omega))}+ \left\|z_0\right  \|_{H^2(\Omega)}  \bigl), 
$$
where $C>0$ is independent of $\epsilon$. 

In fact, multiplying \eqref{sus-regu} by $\epsilon \Delta z_t$ and integrating over $\Omega$, we get 
$$
\epsilon^2 \int_{\Omega} |\nabla z_t|^2dx + \frac{\epsilon}{2}\frac{d}{dt}\int_{\Omega}|\Delta z|^2dx + \frac{\epsilon}{2}\frac{d}{dt}\int_{\Omega}|\nabla z|^2dx \leq \int_{\Omega} |\nabla f|^2 dx.
$$ 
This last inequality gives $\epsilon z_t \in L^2(0,T; H^1(\Omega))$. Using elliptic regularity for \eqref{sus-regu}, the result follows. 

\end{remark}

\section{Uniform exact controllability to the trajectory}\label{sec4}

In this section we give the proof of Theorem \ref{mainresultt} using similar arguments to those employed, for instance, in \cite{Ima-1}. We will see that the results obtained in the previous section allow us to locally invert the nonlinear system \eqref{system}. In fact, the regularity deduced for the solution of the linearized system \eqref{system-linear-4} will be sufficient to apply a suitable inverse function theorem (see Theorem \ref{teoremadecontrol}  below).

Thus, let us set $u = M_1 + z$ and $v = M_2 + w$ and let us use these equalities in \eqref{system}. We find:
\begin{equation}\label{IFS}
\left |   
\begin{array}{ll}
L(z,w) =  ( -\nabla \cdot (z\nabla w), g \chi) &     \mbox{in}  \  \ Q, \\
\frac{\partial z}{\partial \nu} = \frac{\partial w}{\partial \nu} = 0     &    \mbox{on}  \  \   \Sigma, \\
z(x,0) = u_0 - M_1;  \ w(x,0) = v_0 - M_2   &    \mbox{in}    \  \  \Omega.
\end{array}
\right. 
\end{equation}
where $L$ was introduced in \eqref{5.2}.

This way, we have reduced our problem to a local null controllability result for the solution  $(z,w)$ to the nonlinear problem \eqref{IFS}. We will use the following inverse mapping theorem (see \cite{Dontchev, Graves}):


\begin{theorem} \label{teoremadecontrol}
Let $E$ and $G$ be two Banach spaces and let  $ \mathcal{A} : E \rightarrow G $ be a continuous function from $E$ to $G$ defined in $B_{\eta}(0)$ for some $\eta >0$ with $\mathcal{A}(0) =0$. Let $\Lambda$ be a continuous and linear operator from $E$ onto $G$ and suppose there exists $C_0 >0$ such that 
\be\label{estimavt}
||e||_E \leq C_0||\Lambda (e)||_G
\ee
and that there exists $\delta < C_0^{-1}$ such that 
  \be\label{strictdiffer-1}
|| \mathcal{A} (e_1) - \mathcal{A} (e_2) -\Lambda(e_1-e_2)|| \leq \delta ||e_1 -e_2||
\ee
whenever $e_1,e_2 \in B_{\eta}(0)$. Then the equation $\mathcal{A}(e) = h $ has a solution $e \in B_{\eta}(0)$ whenever $||h||_G \leq c \eta$, where $c =  M^{-1} -\delta$.
\end{theorem}

\begin{remark}
In the case where $ \mathcal{A} \in C^1(E;G)$,using the mean value theorem,  it can be shown, that for any $\delta < M^{-1}$, inequality \eqref{strictdiffer-1} is satisfied with $\Lambda = \mathcal{A}'(0)$ and $\eta >0$ the continuity constant at zero, i. e., 
\be\label{constantuniform1}
||\mathcal{A}'(e) -\mathcal{A}'(0)||_{\mathcal{L}(E;G)} \leq \delta
\ee
whenever $||e||_E \leq \eta$.
\end{remark}

%
%
%
%
%
%
In our setting, we use this theorem with the space $E$ and 
$$
G = X  \times Y,
$$
where

\begin{align}\label{X}
X= \{ (h_1, h_2); & \   e^{-s\hat{\beta}}\hat{\gamma}^{-3/2}h_1 \in L^2(Q),  e^{-s\beta}\gamma^{-1}h_2 \in L^2(0,T;H^1(\Omega))  \nonumber \\
&  \text{and} \  \int_\Omega h_1(x,t) dx = 0 \ \text{a. e.}  \ t \in (0,T)  \},
\end{align}
\begin{align}
Y = \{ (z_0,w_0) \in H^1(\Omega) \times H^2(\Omega); \  \int_\Omega z_0dx=0 \  \text{and} \  \frac{\partial w_0}{\partial \nu} =0 \ \text{on} \ \partial \Omega \}
\end{align}

and, for each $0 < \epsilon \leq 1$,  the operator 
$$
\mathcal{A}(z,w,g) =(L(u,v) + ( ( \nabla \cdot (z\nabla w), -g \chi) ), z(.,0), w(.,0) ) \  \forall (z,w,g) \in E.
$$

We have 
$$
\mathcal{A}'(0,0,0) =(L(u,v) + ( 0, -g \chi) ), z(.,0), w(.,0)) \  \forall (z,w,g) \in E.
$$

In order to apply  Theorem \ref{teoremadecontrol}  to our problem, we must check that the previous framework fits the regularity required. This is done using the following proposition.

\begin{proposition}\label{AC1}
$\mathcal{A} \in C^1(E;G)$.
\end{proposition}

\begin{proof}
All terms appearing in  $\mathcal{A}$ are linear (and consequently $C^1$), except for the term $\nabla \cdot (z\nabla w)$. However, the operator 
\be\label{bilinear-fe}
\bigl((z_1,w_1,g_1), (z_2,w_2,g_2) \bigl) \mapsto \nabla \cdot (z_1\nabla w_2)
\ee
is bilinear, so it suffices to prove its continuity from $E\times E$ to $X$.

In fact, we have
\begin{align}
\left\| \nabla \cdot (z_1\nabla w_2) \right \|_X & \leq  \overline{C}\left\|e^{-s\hat{\beta}}\hat{\gamma}^{-3/2} z_1\nabla w_2 \right \|_{L^2(0,T; H^1(\Omega))}  \\
& \leq \overline{C}\left\| e^{s/2\beta^* - s\hat{\beta}} \hat{\gamma}^{13/8} z_1\right \|_{L^{\infty}(0,T; H^1(\Omega))} \left\|  e^{-s/2\beta^*} \hat{\gamma}^{-25/8} w_2 \right \|_{L^2(0,T; H^3(\Omega))},\nonumber
\end{align}
for a positive constant $\overline{C}$ which does not depend on $\epsilon$.

Therefore, continuity of \eqref{bilinear-fe} is established and the  proof Proposition \ref{AC1} is finished.
\end{proof}

An application of  Theorem \ref{teoremadecontrol}  gives the existence of $\delta, \eta > 0$, which a priori depend on $\epsilon$,   such that  if $||(u_0 - M_1,v_0 - M_2)|| \leq \eta/(C_0^{-1}-\delta)$, then there exists a control $g = g(\epsilon)$ such that the associated solution $(z,w)$ to \eqref{IFS} verifies $z(T) = w(T) = 0$ and $||(z, w, g)||_E \leq \eta$.  To finish the proof of  Theorem \ref{mainresultt}, we must show that $C_0$, $\eta$ and $\delta$ does not depend on $\epsilon$. This is a direct consequence from the fact that the constant $C_0$ in \eqref{estimavt}  does not depend on $\epsilon$ (see Theorem  \ref{UnifResult}),  that we can take any  $\delta < C_0^{-1}$ and that $\eta$ can be chosen to be $\delta/\overline{C}$.

\section*{Acknowledgements}
The first named  author thanks the Laboratoire Jacques-Louis Lions for arranging his stay and the kind hospitality during the visit, which was useful for completing this paper. This work is partially supported by the ERC advanced grant 266907 (CPDENL) of the 7th Research Framework Programme (FP7). F. W. Chaves-Silva has been supported by the Grant  BFI-2011-424 of the Basque Government and partially supported by the Grant  MTM2011-29306-C02-00 of the MICINN, Spain, the ERC Advanced Grant FP7-246775 NUMERIWAVES, ESF Research Networking Programme OPTPDE and the Grant PI2010-04 of the Basque Government.

%
%

\appendix

\section{Some technical results}\label{ApA}
In this appendix we prove some technical results  used in the proof of Theorem \ref{Theo-1}.
\begin{lemma}\label{transpocart}
There exist $C = C(\Omega, \omega)$ and  $\lambda_0 = \lambda_0(\Omega, \omega)$ such that,  for every $\lambda \geq \lambda_0$, there exists $s_0 = s_0(\Omega, \omega, \lambda)$ such that, for any $s \geq s_0(T^4 + T^8)$, $g \in C^\infty_0(Q)$ and $\varphi_T \in C^{\infty}_0(\Omega)$, the solution $\varphi$ of 
\begin{equation}\label{transcar}
\left |   
\begin{array}{ll}
-\varphi_{t}  - \Delta \varphi  = \Delta g &     \mbox{in}  \  \ Q,  \\
\frac{\partial  \varphi}{\partial \nu} = 0     &    \mbox{on}  \  \   \Sigma, \\
\varphi(x,T) =  \varphi_T   &    \mbox{in}    \  \  \Omega, \\
\end{array}
\right. 
\end{equation}
satisfies
\begin{align}\label{car-transpo}
s^3 \iint\limits_{Q} e^{2s\alpha}\phi^3 | \varphi|^2 dxdt  \leq     C \left(  s^3 \iint\limits_{\omega \times (0,T)} e^{2s\alpha}\phi^3 | \varphi|^2 dxdt + s^4 \iint\limits_{Q} e^{2s\alpha}\phi^4 |g|^2 dxdt\right ).
\end{align}
\end{lemma}


\begin{proof}  The proof is inspired by the arguments in \cite{FC-G-Trans} (see also \cite{C-G-B-P-1, IY}).

We view $\varphi$ as a solution by transposition of  \eqref{transcar}. This means that $\varphi$ is the unique function in $L^2(Q)$ satisfying
\begin{equation}\label{transp}
 \iint\limits_{Q} \varphi h dx dt= \iint\limits_{Q} g \Delta z dxdt +\int_{\Omega} \varphi_T(x) z(x,T) dx \  \forall h\in L^2(Q),
\end{equation}
where we have denoted by $z$ the solution of the following problem:
   $$
   \left| 
    \begin{array}{ll} z_t-\Delta z=h&\mbox{ in }Q,
   \\
   \noalign{\smallskip} \displaystyle \frac{\partial z}{\partial n}=0&\mbox{ on
   }\Sigma,
   \\
   \noalign{\smallskip} z(x,0)=0&\mbox{ in }\Omega.
   \end{array}\right.
   $$
   
Let us introduce the space
   $$
X_0=\left \{ z \in C^2(\overline Q): \frac{\partial z}{\partial n}=0 \ \hbox{on} \ \Sigma \right \},
   $$
the operators $\mathcal{ L}=\partial_{t}-\Delta$,  $\mathcal{ L}^*=-\partial_t-\Delta$ and the norm $\left \|\cdot \right \|_X$, with
   $$
\left \|y\right  \|_X^2= \iint\limits_Q e^{2s\alpha}|\mathcal{L}^*y|^2 dxdt+s^3 \iint\limits_{\omega \times (0,T)}e^{2s\alpha}  \phi^3|y|^2dxdt
   $$
for all $y\in X_0$.
 
   Due to lemma \ref{lemma-4}, $\left \| \cdot \right \|_X$ is indeed a norm in $X_0$.    Let $X$ be the completion of $X_0$ for the norm $||\cdot||_X$.    Then $X$ is a Hilbert space for the scalar product
$(\cdot,\cdot)_X$, with
   $$
(w,y)_X= \iint\limits_Qe^{2s\alpha}(\mathcal{L}^*w)(\mathcal{L}^*y)dxdt+s^3\iint\limits_{\omega \times (0,T)} e^{2s\alpha} \phi^3wydxdt.
   $$
Let us also consider 
   $$
l(w)=s^3 \iint\limits_Qe^{2s\alpha} \phi^3 \varphi wdxdt \ \  \forall w\in X.
   $$
 
 By virtue of lemma \ref{lemma-4}, we have that $ l  \in X'$. Consequently, from the   Lax-Milgram's lemma,  there exists a unique  $y \in X$ such that 
 $$
 (y,w)_X = l (w), \ \forall w \in X.
 $$

Now, let us set 
   \begin{equation}\label{optimaleq-z}
   \hat{v}=-s^3 e^{2s\alpha}  \phi^3 y1_{\omega} \ \mbox{and} \ \hat{z}=e^{2s\alpha}\mathcal{L}^*y.
   \end{equation}
It is not difficult to see that $\hat{z}$ is, together with $\hat{v}$, a solution to the null controllability problem
\begin{equation}\label{optimaleq}
\begin{cases}
\hat{z}_t  - \Delta \hat{z}  =  s^3e^{2s\alpha} \phi^3\varphi+\widehat v1_{\omega}  &     \mbox{in}  \  \ Q,  \\
\frac{\partial \hat{z}}{\partial \nu } = 0    &    \mbox{on}  \  \   \Sigma, \\
\hat{z}(0) = \hat{z}(T) = 0  &    \mbox{in}    \  \  \Omega .
\end{cases}
\end{equation}
 We have 
\begin{align}\label{ineq1-z}
\qquad\displaystyle \left \|y  \right  \|^2_{X} &= \iint\limits_Q e^{-2s\alpha}|\widehat z|^2dxdt+s^{-3}\iint\limits_{\omega \times (0,T)}e^{-2s\alpha} \phi^{-3}|\widehat v|^2dxdt  \leq C s^3 \iint\limits_Qe^{2s\alpha} \phi^3|\varphi|^2 dxdt,
\end{align}
for $\lambda\geq \lambda_0$ and  $ s \geq  s_0(T^4+T^8)$,
since
$$
\left \| l \right \|_{X'}\leq Cs^{3/2} \bigl( \iint\limits_Qe^{2s\alpha} \phi^3 |\varphi|^2 dxdt\bigl)^{1/2}
$$
for this choice of the parameters $s$ and $\lambda$. 
 
 From \eqref{transp} and \eqref{optimaleq}, it follows that 
   \begin{equation}\label{hattransp}
   \begin{array}{c} \displaystyle
   s^3 \iint\limits_Qe^{2s\alpha} \phi^3|\varphi|^2dxdt = \iint\limits_Q g \Delta \widehat zdxdt -\iint\limits_{\omega \times (0,T)}\varphi \widehat v dxdt.
   \end{array}
   \end{equation}
From \eqref{hattransp}, we see that  the proof of   \eqref{car-transpo} is completed if we  bound $\Delta \widehat z$ in $Q$  in terms of the left-hand side of $(\ref{hattransp})$. In order to do that, we need the following estimate.
 
 {\sc Claim 1.}  
 For $\lambda\geq \lambda_0$ and $ s\geq  s_0(T^4+T^8)$, the following estimate holds
 \begin{align}\label{IYineq}
 s^{-2} \iint\limits_{Q}   e^{-2s\alpha}\phi^{-2}|\nabla\widehat z|^2 dxdt &+ \iint\limits_{Q}  e^{-2s\alpha}|\widehat z|^2dxdt    +s^{-3}\iint\limits_{\omega \times (0,T)}\phi^{-3}e^{-2s\alpha}|\widehat v|^2 dxdt \nonumber \\
 &  \leq Cs^3 \iint\limits_{Q} e^{2s\alpha} \phi^3|\varphi|^2dxdt.
   \end{align}
\textit{Proof of Claim 1.}  \  In order to get an estimate of $|\nabla\widehat z|^2$, we multiply \eqref{optimaleq} by $s^{-2}e^{-2s\alpha}\phi^{-2}\,\widehat z$.  Integration by parts with respect to $x$  gives 
   \begin{align}\label{z-u-eq}
& s^{-2} \iint\limits_{Q} e^{-2s\alpha} \phi^{-2}\widehat z {\widehat z}_t dxdt+  s^{-2} \iint\limits_{Q} e^{-2s\alpha} \phi ^{-2} |\nabla\widehat z|^2 dxdt  \nonumber \\
   &-2s^{-1} \lambda \iint\limits_{Q}e^{-2s\alpha} \phi^{-1}   \nabla\eta^0\cdot\nabla\widehat z \widehat z dxdt -2s^{-2}\lambda  \iint\limits_{Q}e^{-2s\alpha} \phi^{-2}    \nabla\eta^0\cdot\nabla\widehat z \widehat z dxdt \nonumber  \\
   & = s  \iint\limits_{Q}  \phi  \varphi \widehat z   +s^{-2} \iint\limits_{\omega \times (0,T)}e^{-2s\alpha}\phi^{-2} \widehat v \widehat z dxdt.
   \end{align}
   Now we  integrate by parts with respect to the time variable in the first term.
   We obtain the following:
   \begin{align}\label{chepas}
 s^{-2}  \iint\limits_{Q} e^{-2s\alpha} \phi^{-2} \widehat z {\widehat z}_t dxdt &= -\frac{1}{2}s^{-2}\iint\limits_{Q}  (e^{-2s\alpha} \phi^{-2})_t |\widehat z|^2 dxdt   \leq C \iint\limits_{Q}   e^{-2s\alpha} |\widehat z|^2 dxdt,
   \end{align}
since 
$$
| (e^{-2s\alpha} \phi^{-2})_t| \leq C s\phi^{-3/4}e^{-2s\alpha} \ \text{for} \ \lambda \geq 1.
$$

   Finally, using Young's inequality for the other terms of (\ref{z-u-eq}),  we obtain
   \begin{align}\label{chepas-2}
   -2s^{-1}\lambda & \iint\limits_{Q}e^{-2s\alpha} \phi^{-1}   \nabla\eta^0\cdot\nabla\widehat z \widehat z dxdt -2s^{-2}\lambda\iint\limits_{Q}e^{-2s\alpha} \phi^{-2}  \nabla\eta^0\cdot\nabla\widehat z \widehat z dxdt
  \nonumber  \\
   &\leq C \iint\limits_{Q}  e^{-2s\alpha}|\widehat z|^2dxdt+ \frac{1}{2}s^{-2}  \iint\limits_{Q}  e^{-2s\alpha} \phi^{-2} |\nabla\widehat z|^2 dxdt,
   \end{align}

   \begin{align}\label{chepas-3}
s  \iint\limits_{Q}  \phi \varphi \widehat z dxdt \leq C\left( \iint\limits_{Q}  e^{-2s\alpha} |\widehat z|^2 dxdt+s^3 \iint\limits_{Q}e^{2s\alpha}\phi^3|\varphi|^2dxdt \right)
   \end{align}
and
   \begin{equation}\label{chepas-4}
 s^{-2}\lambda^{-2}\iint\limits_{\omega \times (0,T)}e^{-2s\alpha}\phi^{-2} \widehat v \widehat zdxdt \leq C\left( \iint\limits_{Q}  e^{-2s\alpha} |\widehat z|^2 dxdt+s^{-3}\lambda^{-4}\iint\limits_{\omega \times (0,T)}e^{-2s\alpha} \phi^{-3} |\widehat v|^2 dxdt\right),
   \end{equation}
since $s^{-1}\phi^{-1} \leq C$.

From \eqref{z-u-eq}, \eqref{chepas}-\eqref{chepas-4} and \eqref{ineq1-z}, Claim $1$ is proved.

 {\sc Claim 2.}  For $\lambda\geq \lambda_0$ and  $s\geq s_0(T^4+T^8)$, the following estimate holds
  \begin{align}\label{IYineq2}
   & s^{-4}  \iint\limits_{Q}  e^{-2s\alpha}\phi^{-4}(|\widehat   z_t|^2+|\Delta\widehat z|^2) dxdt + s^{-2}  \iint\limits_{Q}  e^{-2s\alpha}\phi^{-2}|\nabla\widehat z|^2 dxdt \nonumber \\ 
 & + \iint\limits_{Q}  e^{-2s\alpha}|\widehat z|^2 dxdt + s^{-3}\iint\limits_{\omega \times (0,T)}e^{-2s\alpha}\phi^{-3}|\widehat v|^2 dxdt  \leq Cs^3 \iint\limits_{Q}   e^{2s\alpha}\phi^3|\varphi|^2 dxdt.
   \end{align}
\textit{ Proof of Claim 2.}   \ We multiply \eqref{optimaleq}  by the function $-s^{-4}e^{-2s\alpha}\phi^{-4}\Delta \widehat z$ and integrate over $Q$.
   We obtain the following:
   \begin{equation}\label{zeq}
   \begin{array}{l}
   \displaystyle s^{-4}  \iint\limits_{Q}  e^{-2s\alpha}\phi^{-4}| \Delta \widehat z|^2 dxdt=s^{-4} \iint\limits_{Q}  e^{-2s\alpha}\phi^{-4}\widehat z_t \Delta\widehat z dxdt
   \\ \noalign{\medskip}\displaystyle
   \hskip0.5cm - s^{-1} \iint\limits_{Q}  \phi^{-1}\Delta \widehat z \varphi dxdt-s^{-4} \iint\limits_{\omega \times (0,T)}e^{-2s\alpha}\phi^{-4} \Delta \widehat z \widehat v dxdt.
   \end{array}
   \end{equation}

   The last two terms in the right hand side can be estimated as
follows:
   \begin{align}\label{esti1}
   s^{-1} \iint\limits_{Q}  \phi^{-1} \Delta \widehat z \varphi dxdt  \leq & \ \frac{1}{4}s^{-4}\iint\limits_{Q}  e^{-2s\alpha}\phi^{-4}|\Delta \widehat   z|^2dxdt +Cs^2 \iint\limits_{Q}   e^{2s\alpha}\phi^2|\varphi|^2dxdt
   \end{align}
and
   \begin{align}\label{esti2}
   s^{-4} \iint\limits_{\omega \times (0,T)} e^{-2s\alpha}\phi^{-4}\Delta \widehat z \widehat v dxdt    \leq & \  \frac{1}{4}s^{-4}  \iint\limits_{Q}   e^{-2s\alpha}\phi^{-4}|\Delta \widehat z|^2dxdt +C s^{-4}\iint\limits_{\omega \times (0,T)} e^{-2s\alpha}\phi^{-4}|\widehat v|^2 dxdt.
   \end{align}

   The last integrals in the inequalities
(\ref{esti1})--(\ref{esti2}) can be easily bounded using (\ref{IYineq}), provided we take $s\geq C T^8$.   Hence
   \begin{align}\label{16p}
  s^{-4}&  \iint\limits_{Q}  e^{-2s\alpha}\phi^{-4}|\widehat z_t|^2 dx dt + s^{-4}\iint\limits_{Q}  e^{-2s\alpha}\phi^{-4}| \Delta \widehat z|^2 dxdt \nonumber \\
   \leq & \ C \bigl(s^{-3} \iint\limits_{\omega \times (0,T)}e^{-2s\alpha}|\widehat v|^2 dx dt + s^{3} \iint\limits_{Q}  e^{2s\alpha}\phi^{3}|\varphi|^2 dx dt\bigl) \nonumber \\
  & + s^{-4} \iint\limits_{Q}  e^{-2s\alpha}\phi^{-4}\widehat z_t \Delta\widehat z dx dt.
   \end{align}

   Let us now deal with the last term in the right hand side of
(\ref{16p}).
   We integrate by parts with respect to $x$ and we get
   \begin{align}\label{eq2}
   s^{-4}  \iint\limits_{Q}  e^{-2s\alpha}\phi^{-4}\widehat z_t \Delta\widehat   z dxdt &=-\frac{1}{2}s^{-4} \iint\limits_{Q}     e^{-2s\alpha}\phi^{-4}\frac{\partial}{\partial t}|\nabla\widehat  z|^2 dxdt \nonumber \\
&-s^{-4} \iint\limits_{Q}    \nabla(e^{-2s\alpha}\phi^{-4})\cdot\nabla\widehat z \widehat   z_t dxdt.
   \end{align}
We  integrate by parts with respect to $t$ in the first term of the right hand side of~(\ref{eq2}). This yields:
   $$
\begin{array}{l}\displaystyle
-\frac{1}{2}s^{-4}  \iint\limits_{Q}  e^{-2s\alpha}\phi^{-4}\frac{\partial}{\partial t}|\nabla\widehat z|^2 dxdt=\frac{1}{2}s^{-4} \iint\limits_{Q}  (e^{ - 2s\alpha}\phi^{-4})_t|\nabla\widehat z|^2 dxdt.
\end{array}
   $$
     Therefore,
   \begin{equation}\label{nablazt}
   -\frac{1}{2}s^{-4} \iint\limits_{Q}  e^{-2s\alpha}\phi^{-4}\frac{\partial}{\partial t}|\nabla\widehat   z|^2 dxdt\leq Cs^{-2}\iint\limits_{Q}   e^{-2s\alpha}\phi^{-2}|\nabla\widehat z|^2 dxdt,
   \end{equation}
since 
  $$
|(e^{-2s\alpha}\phi^{-4})_t|\leq Cs^{5/4}e^{-2s\alpha}\phi^{-15/4} \ \text{if} \ s\geq CT^8.
   $$

   In order to estimate the second term in (\ref{eq2}), we take into account that
   $$
|\nabla(e^{-2s\alpha}\phi^{-4})|\leq C s  e^{-2s\alpha} \phi^{-3}
   $$
and  use Young's inequality to  obtain
   \begin{align}\label{chepas2}
s^{-4}  \iint\limits_{Q}  (\nabla(e^{-2s\alpha}\phi^{-4})\cdot\nabla\widehat z)\widehat z_t dxdt & \leq \frac{1}{4}s^{-4} \iint\limits_{Q}  e^{-2s\alpha}\phi^{-4}|\widehat z_t|^2 dxdt \nonumber \\
&+Cs^{-2} \iint\limits_{Q}  e^{-2s\alpha}\phi^{-2}|\nabla\widehat z|^2 dxdt.
   \end{align}
   From (\ref{16p})--(\ref{chepas2}), Claim $2$ is proved.

 Let us now finish the proof of Lemma \ref{transpocart}. From identity  \eqref{hattransp}, we have 
 \begin{align}\label{hattransp-1}
   s^3\lambda^4 \iint\limits_Qe^{2s\alpha} \phi^3|\varphi|^2dxdt &\leq C\biggl( s^3 \iint\limits_{\omega \times (0,T)} e^{2s\alpha}\phi^3 | \varphi|^2 dxdt + s^4 \iint\limits_{Q} e^{2s\alpha}\phi^4 |g|^2 dxdt\biggl) \\
   &+ \delta \biggl( s^{-4} \iint\limits_{Q}  e^{-2s\alpha}\phi^{-4}| \Delta \widehat z|^2 dxdt  + s^{-3}\iint\limits_{\omega \times (0,T)}e^{-2s\alpha}\phi^{-3}|\widehat v|^2 dxdt\biggl),  \nonumber
   \end{align}
for any $\delta >0$. 

Finally, from Claim $2$, the proof of Lemma  \ref{transpocart} is finished.

\end{proof}

Now we prove lemma \ref{estimate-1},  used in the proof of Theorem \ref{Theo-1}.

\begin{proof}[Proof of Lemma \ref{estimate-1}]
Given $h \in L^2(Q)$, let $z$ be the unique solution of 
\begin{equation}\label{z}
\left |   
\begin{array}{ll}
z_{t}  - \Delta z = h&     \mbox{in}  \  \ Q,  \\
z = 0     &    \mbox{on}  \  \   \Sigma, \\
z (x,0) = 0,   &    \mbox{in}    \  \  \Omega. \\
\end{array}
\right. 
\end{equation}
By standard energy estimates, we have 
$$
\left \|z_t\right \|_{L^2(Q)}^2 +\left \| \Delta z\right  \|_{L^2(Q)}^2 + \left \| \nabla z\right  \|^2_{L^{\infty}(0,T; L^2(\Omega))}\leq C\left \| h \right \|_{L^2(Q)}^2.
$$
The duality between  \eqref{eta} and \eqref{z} gives 
\be \label{E-0}
 \iint\limits_{Q} \eta hdxdt = \iint\limits_{Q} (a \theta \Delta \xi   + \theta \Delta f_1)z dxdt.
\ee
Integrating by parts, we have 
\begin{align}\label{E-1}
 \iint\limits_{Q}  \theta \Delta \xi  z dxdt = \iint\limits_{Q}  \Delta \theta  \xi  z dxdt + 2 \iint\limits_{Q}  \nabla \theta \cdot \nabla  z \xi dxdt + \iint\limits_{Q}  \theta  \xi  \Delta z dxdt
\end{align}
and
\begin{align}\label{E-2}
 \iint\limits_{Q}  \theta \Delta f_1  z dxdt  = \iint\limits_{Q}  \Delta \theta  f_1  z dxdt + 2 \iint\limits_{Q}  \nabla \theta \cdot \nabla  z f_1 dxdt + \iint\limits_{Q}  \theta  f_1  \Delta z dxdt.
\end{align}
The result follows  from  \eqref{E-0} with $h = \eta$ and the fact that 
\be\label{weightestimate-s}
|\Delta \theta| \leq Cs^5\phi^5e^{s\alpha} \ \text{and} \ |\nabla \theta| \leq Cs^4\phi^4e^{s\alpha} \ \ \text{in} \ Q. 
\ee

\end{proof}

\end{document}